\documentclass[12pt]{amsart}
\usepackage{amsmath,amsthm,amsfonts,amssymb,eucal}

\renewcommand {\a}{ \alpha }

\newcommand{\m}{\mu}
\newcommand{\vare}{\varepsilon}
\newcommand{\g}{\gamma}
\newcommand{\G}{\Gamma}

\newcommand{\varf}{\varphi}
\renewcommand{\d}{\delta}

\newcommand{\D}{\Delta}
\newcommand{\s}{\sigma}
\newcommand{\Sg}{\Sigma}
\renewcommand{\l}{\lambda}

\newcommand{\Om}{\Omega}

\newcommand{\R}{ \mathbb R}

\newcommand{\N}{ \mathbb N}

\newcommand{\Z}{ \mathbb Z}

\newcommand{\CD}{\mathcal D}
\newcommand{\CE}{\mathcal E}

\newcommand{\CG}{\mathcal G}
\newcommand{\CH}{\mathcal H}

\newcommand{\CV}{\mathcal V}

\newcommand {\GS}{\mathfrak S}

\newcommand {\ba}{\mathbf a}

\newcommand {\bb}{\mathbf b}

\newcommand {\BB}{\mathbf B}

\newcommand {\BH}{\mathbf H}
\newcommand {\BI}{\mathbf I}

\newcommand {\BM}{\mathbf M}
\newcommand {\BN}{\mathbf N}

\newcommand {\BS}{\mathbf S}
\newcommand {\BT}{\mathbf T}

\newcommand{\wt}{\widetilde}
\newcommand{\wh}{\widehat}

\DeclareMathOperator{\dom}{Dom} 
\DeclareMathOperator{\re}{Re} \DeclareMathOperator{\dist}{dist}
\DeclareMathOperator{\res}{\restriction}

\DeclareMathOperator{\rank}{rank}

\newtheorem{thm}{Theorem}[section]

\newtheorem{prop}[thm]{Proposition}

\theoremstyle{definition}

\theoremstyle{remark}

\newtheorem{assump}[thm]{Assumption}
\numberwithin{equation}{section}

\newcommand{\thmref}[1]{Theorem~\ref{#1}}

\newcommand{\bsymb}{\boldsymbol}

\newcommand{\sh}{Schr\"odinger }
\newcommand{\vs}{\vskip0.2cm}
\newcommand{\pl}{\rm{pl}}

\newcommand{\0}{\bsymb{0}}
\newcommand{\1}{\bsymb{1}}
\begin{document}

\title[Sparse potentials on graphs]{Spectral estimates for the Schr\"odinger operators with sparse potentials on graphs}

\author[G. Rozenblum]{G. Rozenblum}
\address{Department of Mathematics \\
                        Chalmers University of Technology\\
                        and  The University of Gothenburg \\
                         S-412 96, Gothenburg, Sweden}
\email{grigori@chalmers.se}
\author[M. Solomyak]{M. Solomyak}
\address{Department of Mathematics
\\ Weizmann Institute\\ Rehovot\\ Israel}
\email{solom@wisdom.weizmann.ac.il}

\subjclass[2010] {47A75; 47B39, 34L15, 34L20}
\keywords{Eigenvalue  estimates, Schr\"odinger operator on graphs,
Dimension at infinity, Sparse potentials.}

\begin{abstract}
A construction of "sparse potentials", suggested in \cite{RS09} for the
lattice $\Z^d,\ d>2$, is extended to a wide class of  combinatorial and metric graphs  whose
global dimension is a number $D>2$. For the Schr\"odinger operator $-\D-\a
V$ on such graphs, with a sparse potential $V$, we study the behavior (as $\a\to\infty$) of the number
$N_-(-\D-\a V)$ of negative eigenvalues of $-\D-\a V$. We show that by means of sparse potentials
one can realize
any prescribed asymptotic behavior of $N_-(-\D-\a V)$ under
very mild regularity assumptions. A similar construction works also for
the lattice $\Z^2$, where $D=2$.
\end{abstract}
\maketitle

\section{Introduction}\label{intro}
There exists a far-reaching parallelism between the theory of the Schr\"odinger operator on $\R^d$
and its discrete analogue on $\Z^d$. Still, there are issues where this parallelism is violated. Certainly, the most known examples of such violation come from the fact that,
unlike the case of $\R^d$, the Laplacian on $\Z^d$ is a bounded operator. However, there are also examples of a different origin.
One of them concerns estimates on the number $N_-(-\D-\a V)$ of the negative eigenvalues for
the operator $-\D-\a V$, where $V\ge0$ is a potential and $\a>0$ is a large parameter ("the coupling constant"). Both for $\R^d$ and for $\Z^d$ the "Rozenblum -- Lieb -- Cwikel estimate" (RLC estimate)
is satisfied: namely, one has
\begin{equation}\label{RLC}
    N_-(-\D-\a V)\le C\a^{d/2}\|V\|_{d/2}^{d/2},\qquad C=C(d);\ d\ge3.
\end{equation}
Here $\|\cdot\|_q$ stands for the $L^q$-norm in the continuous case and for the $\ell^q$-norm in the discrete case. In both cases the inequality \eqref{RLC} is valid for any potential $V$ such that the
norm in the right-hand side is finite. The difference between these cases is this:
\begin{equation*}
  \mbox{ on }\R^d: \   N_-(-\D-\a V)=o(\a^{d/2}) \ \mbox{if and only if } V\equiv0,
\end{equation*}
while
\begin{equation}\label{oZ}
    \mbox { on }\Z^d: \   N_-(-\D-\a V)=o(\a^{d/2}) \ \mbox { for any } V\in{\ell^{d/2}.}
\end{equation}

This latter effect was, probably, observed for the first time in \cite{RS08}.
The same effect manifests itself, if instead of the \sh operator on
$\Z^d$ we consider the \sh operator on an arbitrary combinatorial or metric graph, whose "global
dimension" is greater than $2$; this was shown in \cite{RS10}. Here the global dimension is defined as the exponent $D$,  such that the heat kernel $P(t;x,y)$, corresponding to the Laplacian on the given graph, satisfies the estimate
\begin{equation}\label{D}
    \|P(t;.,.)\|_{L^\infty}=O(t^{-D/2}),\qquad t\to\infty.
\end{equation}
For the graph $\Z^d$ one has $D=d$.

For the discrete case, in view of \eqref{oZ}, even a single example of a potential such that
$N_-(-\D-\a V)\asymp \a^{d/2}$ had been unknown for some time. Only in \cite{RS09} the authors suggested a general scheme that allowed us (for $G=\Z^d$) to construct discrete potentials, such that the function $N_-(-\D-\a V)$ has any prescribed asymptotic behavior as $\a\to\infty$, under
very mild regularity conditions. This includes also the case $N_-(-\D-\a V)\sim C\a^{d/2}$ with $C>0$. This scheme
was based upon some asymptotic estimates for $N_-(-\D-\a V)$ whose nature differs drastically from the estimate \eqref{RLC} and from
the majority of other estimates described in \cite{RS08,RS10}: namely, these latter ones  are permutation-invariant with respect to
the potential $V$. In contrast, the new estimates and asymptotic formulas obtained in \cite{RS09} depend not only on the size of $V$ but also on the geometry of its support. More exactly, the potential $V$ is supposed to be supported on a sequence of very sparsely placed  points. The eigenvalue distribution   results obtained for such \emph{sparse potentials}  are quite flexible.  On the other hand, this  class of potentials is rather special.

In the present paper we show that the construction of sparse potentials \emph{with a prescribed eigenvalues behavior}  is not restricted  to $G=\Z^d, \ d>2$, but can be, with minor modifications, extended  to a wide class of combinatorial graphs with $D>2$. This we do in Section \ref{comb}.

The case $D=2$ is critical in this topic, the general scheme of the spectral analysis of the \sh operator breaks down at many  points (as it also happens to the classical \sh operator in $\R^2$). In Section \ref{Z2setting} we show, however, that the approach based upon sparse potentials works for the important case $G=\Z^2$ as well (although using some more specific analytical tools.)

Finally, in Section \ref{metric} we show that this approach can be successfully modified to the \sh operator on such metric graphs, for which the associated weighted combinatorial graph admits the  construction of sparse potentials.

Sparse potentials, in the spectral theory of the one-dimensional Schr\"o\-dinger  operator, were introduced by D.Pearson, \cite{Pearson}. Sparse potentials in higher dimensions were considered by S.Molchanov and B.Vainberg in \cite{MM1}, \cite{MM2}. Compared with our papers, these authors  analyzed problems of a different type, and they used another definition of sparseness. They did not discuss sparse potentials on graphs different from $\Z^d$.

\section{Operators on combinatorial graphs}\label{comb}
\subsection{Operator $\BB_{V,G}$.}\label{bbvg}
Let $G$ be a combinatorial graph, with the set of vertices $\CV=\CV(G)$ and the set of edges
$\CE=\CE(G)$. We write $e=(v,v')$ for the edge whose endpoints are the vertices $v,v'$.
We then also say that $v,v'$ are neighboring vertices, and write $v\sim v'$.
For simplicity, we assume that the graph is connected and has no loops, multiple edges, or vertices of degree one.  We suppose that $\#\CV=\infty$, and that degrees of all vertices
are finite. On $\CV$ we consider the standard counting measure $\s$. So, $\s(v)=1$ for each vertex $v$. Our basic Hilbert space is $\ell^2=\ell^2(G)=L^2(\CV;\s)$. We also denote $\ell^q=\ell^q(G)=L^q(\CV;\s),\ 0<q\le\infty$.

With each edge $e\in\CE$ we associate a weight $g_e>0$.
On $\ell^2(G)$ we consider the quadratic form
\begin{equation}\label{form}
    \ba_G[f]=\sum_{e\in\CE; e=(v,v')}g_e|f(v)-f(v')|^2,
\end{equation}
with the natural domain
\begin{equation*}
    \dom\ba_G=\{f\in\ell^2(G):\ba_G[f]<\infty\}.
\end{equation*}
This quadratic form is non-negative and closed; recall that by  definition the latter means that $\dom\ba_G$ is complete
with respect to the norm $(\ba_G[f]+\|f\|^2)^{1/2}$.
Denote by $-\D_G$ the  self-adjoint operator on $\ell^2$, associated with the quadratic form $\ba_G$. It is easy to see  that on its domain the operator $\D_G$ acts according to the formula
\begin{equation}\label{oper}
    (\D_G f)(v)=\sum_{v'\sim v}f(v') g_{(v,v')}-f(v)\sum_{v'\sim v} g_{(v,v')}.
\end{equation}
If the weights $g_e$ and the degrees $\deg v$ of all vertices are uniformly bounded, then the quadratic form \eqref{form} and, hence, the operator $\D_G$ are bounded.

Due to the embeddings $\ell^1\subset\ell^2\subset\ell^\infty$, any operator, bounded in $\ell^2$, is bounded also as acting from $\ell^1$ to $\ell^\infty$, with the corresponding
estimate for the norms. This applies, in particular, to the operators $\exp(\D_G t)$ and,
therefore, the heat kernel satisfies the estimate $\|P(t;.,.)\|_{L^\infty}\le C<\infty,\ \forall t\in\R$. Our main assumption here is that \eqref{D} is satisfied with some $D>2$. Then by the Varopoulos theory, see \cite{Varop}, we have
\begin{equation}\label{varop}
    \|f\|^2_{\ell^p}\le C\ba_G[f],\qquad p=p(D)=D(D-2)^{-1},\ \forall f\in\dom\ba_G.
\end{equation}
In particular, \eqref{varop} is valid for any function with finite support.
Denote by $\CH(G)$ the completion of the set of all such functions in the metric generated
by the quadratic form $\ba_G$. This is a Hilbert space; the estimate \eqref{varop}
implies that $\CH(G)$ can be realized as a space of functions, embedded into $\ell^p$. It follows from here that
any function $f\in\CH(G)$ tends to zero "at infinity" (notation $f\to0$). More exactly,
$f\to0$ means that for any $\vare>0$ the set $\{v\in\CV:|f(v)|>\vare\}$ is finite.

\vs

Below we will systematically use the Birman -- Schwinger operator $\BB_{V,G}$ that corresponds to the
\sh operator  $-\D-\a V$,  see, e.g., \cite{RS08} and the references therein.
Recall that in our case $\BB_{V,G}$
 is the operator in the space $\CH(G)$, generated by the quadratic form
 \begin{equation}\label{bbV}
    \bb_{V,G}[f]=\sum_{v\in\CV}V(v)|f(v)|^2.
 \end{equation}
 The eigenvalue distribution characteristics for the operators $-\D_G-\a V$ in $\ell^2(G)$ and $\BB_{V,G}$ in $\CH(G)$ are connected by the famous Birman -- Schwinger principle: for any $\a>0$ one has
 \begin{equation}\label{bsch}
    N_-(-\D_G-\a V)=n(\a^{-1},\BB_{V,G}).
 \end{equation}
 Here $n(s,\BT)=\#\{n\in\N:\l_n(\BT)>s\},\ \ s>0,$ stands for the eigenvalue distribution function of a compact non-negative operator $\BT$.
 Moreover, the next two properties are equivalent:

 (A) The operator $\BB_{V,G}$ is compact;

 (B) $N_-(-\D_G-\a V)<\infty$ for all $\a>0$.
\vs
 The equality \eqref{bsch} allows one to formulate spectral estimates and asymptotic formulas for
 $-\D_G-\a V$ in the terms of the operator $\BB_{V,G}$. For instance, the estimate \eqref{RLC} means that for $V\in\ell^{d/2}(\Z^d)$ the operator
 $\BB_{V,G}$ belongs to the "weak Neumann -- Schatten ideal" $\Sg_{d/2}$, and that its norm in this ideal satisfies the estimate
 \begin{equation*}
    \|\BB_{V,G}\|_{\Sg_{d/2}}\le C\|V\|_{d/2};
 \end{equation*}
 see \cite{BSbook}, \S 11.6, for the definition and basic properties of the classes $\Sg_q$.
\vs
 Below we describe the estimates of a completely different nature. They are valid for a rather wide class of graphs and for a special, but rather restricted class of potentials. To describe these estimates, we need some properties of the Green function for the Laplacian $\D_G$. Recall that we always assume $D>2$.

 We denote by $(.,.)_{\CH(G)}$ the scalar product in $\CH(G)$, so that
\[(f,g)_{\CH(G)}=\ba_G[f,g].\]
Up to the end of this section, we drop the subindex $\CH(G)$ in this notation and in the notation of the corresponding norm.

\subsection{The Green function}\label{green}
For a fixed $v\in \CV$ consider the linear functional on $\CH(G)$:
$$\varf_v(f)=f(v).$$
In view of \eqref{varop} this functional is continuous on  $\CH(G)$, and therefore  there exists a unique function $h_v\in \CH(G)$ such that
\begin{equation}\label{1}
    (f,h_v)=f(v), \qquad \forall f\in \CH(G).
\end{equation}
Thus the function $h_v$ satisfies the equation
\begin{equation}\label{2}
    -\D h_v(w)=\d^v_w
\end{equation}
and is, obviously, real-valued. It is natural to call it {\it the Green function} for the operator $\D_G$.
Taking  $f=h_w$ in \eqref{1}, we obtain
\begin{equation}\label{3m}
   (h_w,h_v)=h_w(v)=h_v(w),\qquad\forall v,w\in \CV.
\end{equation}
By \eqref{varop}, this implies that
\begin{equation}\label{3n}
    h_w(v)\to0\qquad{\rm{as} }\ w\ \rm{is\ fixed.}
\end{equation}

 For $w=v$  the equality \eqref{3m} gives
    \begin{equation*}
    (h_v,h_v)=h_v(v),
    \end{equation*}
whence all the numbers $\mu_v^2=h_v(v)$ are positive.
We normalize (in $\CH(G)$) the functions $h_v$, i.e., we set $\wt{h}_v=\mu_v^{-1}h_v$.

The function $\mu(v)^{-1}\wt{h}_v$ can be also defined as the {\it unique} function $f\in\CH(G)$ that minimizes the quadratic functional $\|f\|^2$ under the condition $f(v)=1$. Note also that the number $\mu_v$ is nothing but the capacity of the one-point set $\{v\}$ with respect to the quadratic form $\ba_G$.

\subsection{Sparse potentials.}\label{sparse potentials}
\vs
For a graph $G$,
we introduce the notion of  a {\it sparse subset} $Y\subset\CV$ in the same way as this was done for $G=\Z^d$
in \cite{RS09}, Section 6.2. Namely, let $Y\subset G$ be an infinite set, and let $\CH_Y$ stand for the subspace in $\CH(G)$ , spanned
by the functions $\{h_y:y\in Y\}$. We say that $Y$  is sparse if in $\CH(G)$ there exists a compact linear operator
$\BT$, such that the operator $ \BI-\BT$ has bounded inverse and the functions $e_v=(\BI-\BT)^{-1}\wt{h}_v$ form an orthonormal (not necessarily complete) system in $\CH(G)$. We do not discuss here the weakly sparse subsets, also introduced in \cite{RS09}.

We cannot claim that every graph $G$ with $D>2$ contains a sparse subset. In Subsection \ref{sps} we shall give a simple condition of a geometric nature, that guarantees existence of such subsets.

We say that a function $V$ on $G$  (the discrete potential) is sparse, if its support $Y_V=\{v\in\CV:V(v)\neq0\}$ is a sparse subset. Due to this definition, the quadratic form \eqref{bbV} can be re-written as
\begin{equation*}
    \bb_{V,G}[f]=\sum_{v\in Y_V}V(v)|(f,\wt{h}_v)|^2=\sum_{v\in Y_V}V(v)|(f,(\BI-\BT)e_v)|^2.
\end{equation*}
If, in addition, $V\to0$, we always enumerate the points $v\in Y_V,\ v=v_n,$ in such a way that the sequence $V(v_n)$ is monotone.

Due to the sparseness of $Y_V$, the system of elements $\wt{h}_v,\ v\in Y_V$, is close to an orthonormal system in $\CH(G)$. Hence, the properties of the operator $\BB_{V,G}$ should be close to those of the diagonal operator  with the entries $V(v)$, $v\in Y_V$. More precisely,  consider the operator
\begin{equation*}
    \BN_V=\sum_{v\in Y_V}\sqrt{V(v)}(\cdot,e_v)e_v.
\end{equation*}
We have
\[\BN_V(\BI-\BT^*)=\sum_{v\in Y_V} \sqrt{V(v)}(\cdot,\wt h_v)e_v,\]
whence
\[ \|\BN_V(\BI-\BT^*)f\|^2_{\CH(G)}=\sum_{v\in Y_V}V(v)
|(f,\wt h_v)|^2_{\CH(G)}=\bb_{V,G}[f].\]
 This means that
\begin{equation}\label{bv}
    \BB_{V,G}=(\BI-\BT)\BN_V^2(\BI-\BT^*).
\end{equation}

The first property of $\BB_{V,G}$ follows from \eqref{bv} immediately, cf. Theorems 6.3 and 6.4 in \cite{RS09}.

\begin{thm}\label{bound-sparse}
Let $V\ge 0$ be a sparse potential on a weighted graph $G$ with global dimension $D>2$. Then
the corresponding Birman-Schwinger operator $\BB_{V,G}$ is bounded if and only if the function $V$ is bounded, and is compact if and only if $V\to0$.

Moreover, \[ C\|V\|_{\infty}\le\|\BB_{V,G}\|\le C'\|V\|_{\infty},\]
where $C=\|(\BI-\BT)^{-1}\|^{-2}$ and $C'=\|\BI-\BT\|^2$. In the case of compactness, the following two-sided estimate, with the same constants $C,C'$, is valid for the eigenvalues $\l_n(\BB_{V,G})$:
\[ CV_n\le \l_n(\BB_{V,G})\le C'V_n.\]
\end{thm}
Recall that in the case of compactness we have, just due to the way of enumeration, $V_n\searrow0$.
\vs

The next result is more advanced. It extends Theorem 6.6 in \cite{RS09} (where it was $G=\Z^d$) to the case of general graphs. The proof survives. Still, we reproduce it here,
to make it possible to read the present paper independently.

We derive this result under an additional condition on the potential $V$: we assume that
the sequence $\{V_n\}$ is {\it moderately varying}. This means that $V_n\searrow 0$ and
$V_{n+1}/V_n\to1$. We shall use a result of M.G. Krein, see
Theorem 5.11.3 in \cite{GOKR}. Below we present its formulation,
restricting ourselves to the situation we need.
\begin{prop}\label{krein} Let $\BH\ge0$ and
$\BS$ be self-adjoint and compact operators. Suppose
$\rank\BH=\infty$ and the sequence of non-zero eigenvalues
$\l_n(\BH)$ is moderately varying. Then for the operator
$\BM=\BH(\BI+\BS)\BH$ one has
\begin{equation*}
    \l_n(\BM)\sim\l_n^2(\BH).
\end{equation*}
\end{prop}

Now we are in a position to prove the following result.
\begin{thm}\label{sppot}
Let $V\ge0$ be a sparse potential, such that the numbers $V_n$
form a moderately varying sequence. Then
\begin{equation*}
    \l_n(\BB_{V,G})\sim V_n.
\end{equation*}
\end{thm}
\begin{proof}
  In view of \eqref{bv} the non-zero spectrum of the operator $\BB_{V,G}$  coincides
with that of the operator
\begin{equation}\label{kr}
    \BM_V:=\BN_V(\BI-\BT^*)(\BI-\BT)\BN_V=\BN_V(\BI+\BS)\BN_V,
\end{equation}
where
\begin{equation*}
    \BS=-\BT-\BT^*+\BT^*\BT.
\end{equation*}

Now we apply Proposition \ref{krein} to the operators $\BH=\BN_V$
and $\BM_V$, given by \eqref{oper} and \eqref{kr} respectively.
All the assumptions of Proposition are evidently satisfied, and we
get the desired result.
\end{proof}

In particular, taking $V(v_n)=n^{-2/D}$, we obtain a potential $V$ such
that the eigenvalues $\l_n(\BB_{V,G})$ asymptotically behave as $n^{-2/D}$.
This solves a problem discussed in Section \ref{intro}.

\subsection{On the existence of sparse subsets.}\label{sps}
Given a vertex $v\in\CV$, let us consider the function $f_v$ that vanishes at all vertices $w\neq v$ and $f_v(v)=1$. By the extremal property of the Green function $h_v$, we have
\[ \mu_v^{-2}=\mu_v^{-2}\|\wt{h}_v\|^2\le\|f_v\|^2=\ba_G[f_v]=\sum_{w\sim v}g_{(v,w)},\]
whence
\[\m_v^2=h_v(v)\ge \left(\sum_{w\sim v} g_{(v,w)}\right)^{-1}.\]
Given a number $R>0$, we say that a vertex $v$ {\it $R$-mild} if
\begin{equation*}
    \sum_{w\sim v} g_{(v,w)}\le R.
\end{equation*}

\begin{prop}\label{exist}
Suppose that a graph $G$ is such that $D>2$ and that, for some $R>0$, $G$ contains infinitely many $R$-mild
vertices. Then $G$ contains a sparse subset.
\end{prop}
\begin{proof}
We fix an infinite subset $G'\subset G$, consisting of $R$-mild vertices. Choose any double sequence
$\{\vare_{mn}\},\ m,n\in\N; \ m\neq n$, of positive numbers, such that $\vare_{mn}=\vare_{nm}$ and $\sum\vare_{mn}^2<R^{-2}$.
Suppose that the vertices $v_1,\ldots,v_{n-1}$ are already chosen (here $v_1$ is arbitrary). Then, in view of \eqref{3n}, we can choose a point $v_n\in G'$ in such a way that
$|(h_{v_k},h_{v_n})|<\vare_{kn}$ for $k=1,\ldots,n-1$. As a result of this inductive procedure, we get a sequence of vertices $\{v_n\}\subset G'$, such that $(h_{v_n},h_{v_n})=\mu_{v_n}^2$ and
$|(h_{v_m},h_{v_n})|<\vare_{mn}$ if $m< n$. By \eqref{3m}, the same inequality holds for $m>n$. Since all the vertices
$v_n$ are $R$-mild, we have $|(\wt{h}_{v_n},\wt{h}_{v_n})|<R\vare_{mn}$.

It follows that the Gram matrix
$\BM=\{(\wt{h}_{v_n},\wt{h}_{v_n})\}$ is such that $\BM-\BI$ is Hilbert-Schmidt and, moreover, $\|\BM-\BI\|_{\GS_2}<1$. By Theorem VI.3.3 in \cite{GOKR}, this implies the sparseness of the set  $Y$, with $\BT\in\GS_2$; see
\cite{RS09}, section 6.2 (where the case of $G=\Z^d$ was analyzed) for more details.  \end{proof}

\section{Operators on $\Z^2$}\label{Z2setting}
It is well-known that a lot of complications arise in the study of the classical Schr\"odinger operator in $\R^2$, compared to the higher-dimensio\-nal case. They are related to the non-signdefiniteness of the fundamental solution, its logarithmic singularity, the absence of the proper sharp Sobolev embedding theorem, and so on.  One of the important consequences here is the absence of a direct analogy of the Birman-Schwinger principle \eqref{bsch}, since the closure  of the space of compactly supported functions in the metric of the Dirichlet integral is not a space of functions any more.

We encounter the same complication if we try to extend the general  reasoning in the previous section to the case of operators on combinatorial graphs with global dimension $D=2$. Similar to the case of $\R^2$, much more specific instruments are needed.

In this section we present an approach  to the eigenvalue analysis of the Schr\"odinger operator on $\Z^2$.

We consider the graph $G$ whose set of vertices $\CV$ is the lattice $\Z^2$ which we understand as embedded in the natural way into $\R^2$. Two vertices are said to be connected with an edge if they are the closest neighbors in $\R^2$, and the weight of any such edge is taken as $1$. It is convenient to identify a vertex $v\in \Z^2$ with its Cartesian co-ordinates
$x=(x_1,x_2)\in \Z^2\subset\R^2$.

\subsection{The Birman-Schwinger principle,  a Hardy type inequality, and the space $\CH_0$}
Following Section \ref{comb}, we introduce the quadratic form
\begin{equation}\label{2.Form}
    \ba_{\Z^2}[f]=\sum_{x,x'\in\Z^2, x\sim x'}|f(x)-f(x')|^2.
\end{equation}
This quadratic form is bounded in $\ell^2=\ell^2(\Z^2)$ and defines the self-adjoint operator $-\D$.  For a bounded real-valued function $V(x)$ on $\Z^2$ and a coupling constant $\a\ge0$,
we consider the Schr\"odinger operator $-\D-\a V$, and our concern is in the study of the behavior of $N_-(-\D-\a V)$ as $\a\to\infty.$

So far so good, however if we try to apply the Birman-Schwinger principle, as we did it above,  we fail. The reason for this lies in the fact that the space $\CH$, the closure of the space of finitely supported functions in the metric \eqref{2.Form}, is not a space of functions any more. In fact, it is easy to construct a family of functions $u_n$ on $\Z^2$,
converging to a nonzero constant point-wise, so that  $\ba_G[u_n]\to0.$ This is also related to the behavior of the heat kernel for $-\D$ on $\Z^2$. Standard calculations by means of the Fourier series show that the heat kernel $P(t;.,.)$ decays exactly as $t^{-1}$ as $t\to\infty$, so the global dimension equals $2$ and the Varopoulos theory does not apply.

In order to handle these inconveniences, it is sufficient to impose the Dirichlet boundary condition at one point, say at $x=\0=(0,0)$.

So, denote by $\ell^2_0$ the (obviously closed) subspace in $\ell^2$ consisting of functions with zero value at $\0$.
The quadratic form \eqref{2.Form} defines a self-adjoint operator $-\D_{0}$ in $\ell^2_0$, and we denote by $-\D_{0}-\a V$ the corresponding Schr\"odinger operator.

Since the subspace $\ell^2_0$ has codimension $1$ in $\ell^2$, the standard application of the variational principle gives

\begin{equation}\label{2.Comparison}
    N_-(-\D_0-\a V)\le N_-(-\D-\a V)\le N_-(-\D_0-\a V)+1,
\end{equation}
therefore, for our aims, it suffices to study the operator $-\D_0-\a V$.

The minor reduction above brings a great advantage:  applying the Birman-Schwinger principle becomes possible now. To explain this, we start by establishing a Hardy type inequality in $\ell^2_0$.

\begin{prop}\label{2.prop.Hardy}
For some constant $C$ and for any function $f\in \ell^2_0$ with compact support,
\begin{equation}\label{2.Hardy}
    \sum_{x\in\Z^2\setminus{(\0)}}|f(x)|^2|x|^{-2}(\log |x|+2)^{-2}\le C\ba_{\Z^2}[f].
\end{equation}
\end{prop}
\begin{proof} The inequality \eqref{2.Hardy} is proved in a  way, similar to the discrete Hardy type inequality in Sect.4 in \cite{RS09}, by interpolating a function on $\Z^2$ to a function on $\R^2$. In all squares in $\R^2$ with vertices in $\Z^2$,  we apply the same piece-wise linear interpolation as in \cite{RS09}.  Additionally, for the four central squares, the  ones having the vertex $\0$, we multiply the interpolating function by the cut-off  $\varphi(|x|)$, which equals $0$ for $|x|<\frac12$ and equals $1$ for $|x|\ge \frac34$. The inequality \eqref{2.Hardy} follows now by applying the classical
 Hardy inequality in $\R^2$  with $\log$ term (valid for functions vanishing in a fixed neighborhood of the origin) to the interpolant.
\end{proof}

We define the space $\CH_0$ as the closure of the set of finitely supported functions in $\ell^2_0$ in the metric \eqref{2.Form}.
It follows from \eqref{2.Hardy}, that on this set the convergence with respect to \eqref{2.Form} implies the convergence in  $\ell^2$ with the weight as in \eqref{2.Hardy}, therefore $\CH_0$ is a space of functions.

Now we are able to define the Birman-Schwinger operator $\BB^{(0)}_{V,\Z^2}$ in the space $\CH_0$ by means of the quadratic form \eqref{bbV}, obviously, closed. By the general Birman-Schwinger principle and \eqref{2.Comparison}, we have
\begin{equation*}
 n(\a^{-1},\BB^{(0)}_{V,\Z^2})  \le N_-(-\D_0-\a V)\le n(\a^{-1},\BB^{(0)}_{V,\Z^2})+1.
\end{equation*}

\subsection{The Green function}\label{Z2Green}

It follows  from  \eqref{2.Hardy} that  the linear  functional $\varphi_v(f)=f(x)$ is continuous in $\CH_0$ and defines a unique function $h_x\in \CH_0$ such that
\begin{equation}\label{2.functional}
    (f,h_x)=f(x), \ \forall f\in\CH_0.
\end{equation}
Thus the function $h_x$ satisfies the equation \eqref{2}; it  is real-valued and will be called, similar to the considerations in Section  \ref{comb}, the \emph{Green function} for the operator $\D$ on $\Z^2$. The relation \eqref{3m} holds too, but at this point the analogy with Section  \ref{comb} stops. The reason for this is that we do not have the inequality \eqref{varop} any more, and the Hardy inequality \eqref{2.Hardy} does not imply that $h_x(y)\to 0$ as $|x-y|\to\infty$.
Moreover, the asymptotic formula \eqref{2.Green.asymp} below shows that
in $\Z^2$-setting this property fails.
Nevertheless, we define, again,  $\wt{h}_x=\m_x^{-1}h_x$, where $\m_x^2=(h_x,h_x)=h_x(x)$.
Again, as it was in the case of graphs with $D>2$ (see the end of Subsection \ref{green}), the function $\mu_x^{-1}\wt{h}_x$ is
the unique minimizer (in the space $\CH_0(\Z^2)$) of the functional \eqref{2.Form} under the
condition $f(x)=1$.

We are going to find an explicit expression for $h_x$.   To do this, we consider  first the fundamental solution for the Laplacian  on $\Z^2$ \emph{without the boundary condition} at $\0$, i.e., the function $\CG(x), \ x\in \Z^2,$ such that $\D \CG(x)=\d_{\0}^x$. Such function, if exists, is not unique, it is defined up to $\Z^2$-harmonic additive term. We are going to deal with a specific function $\CG$.

For $\Z^d$ with $d > 2$ such solution can be easily found by means of the
Fourier series, see the formula (6.1) in \cite{RS09}.
 In the 2-dimensional case, however, the integral in \cite{RS09}, (6.1), diverges, so a different approach is needed.

The first formula for $\CG(x)$ was found in \cite{McW}. We do not reproduce it since it is not convenient for further calculations.  We are going to use another formula for the fundamental solution, presented, for example, in the book \cite{Spitzer}, formula (1) in Chapter 15:
\begin{equation}\label{2.Fundamental}
    \CG(x)=(2\pi)^{-2}\int\limits_{(-\pi,\pi)^2}\frac{\cos(x\theta)-1}{Z(\theta)}d\theta,
\end{equation}
where
\begin{equation*}
  Z(\theta)=Z(\theta_1,\theta_2)  =1-\frac12(\cos\theta_1+\cos\theta_2), \ x\theta = x_1\theta_1+x_2\theta_2.
\end{equation*}
For our aims it is more  convenient to use the exponential form of \eqref{2.Fundamental}:
\begin{equation}\label{2.FundamentalCompl}
    \CG(x)=(2\pi)^{-2}\int\limits_{(-\pi,\pi)^2}\frac{\exp(ix\theta)-1}{Z(\theta)}d\theta.
\end{equation}

So, $\CG(\0)=0$;
it is also known, see \cite{Spitzer}, Chapter 15, that $ \CG(x)$ has the asymptotics
\begin{equation}\label{2.asympt}
     \CG(x)\sim \frac{1}{\pi} \bigl(2 \log |x| + \log 8 + 2\g\bigr), \ |x|\to\infty,
\end{equation}
uniformly in all directions ($\g$ is the Euler constant).

Now we construct the function $H_x(y)$, as follows:
\begin{equation}\label{2.Green.expl}
    H_x(y)=-\CG(x-y)+\CG(x)+\CG(y).
\end{equation}
One can easily check that $H_x(y)=H_y(x)$, $H_x(\0)=0$ and $-\D_yH_x(y)=\d_x^y$. It will be
shown below that $H_x=h_x$.

By the asymptotics \eqref{2.asympt},
\begin{equation}\label{2.Green.asymp}
    H_x(y)\sim \frac{1}{\pi} \bigl(2\log|x|+ \log 8 + 2\g\bigr), \ |y|\to\infty
\end{equation}
and
\begin{equation}\label{2.Green.asymp.diag}
    H_x(x)=2\CG(x)\sim \frac{2}{\pi} \bigl(2\log|x|+ \log 8 + 2\g\bigr), \ |x|\to\infty
\end{equation}

The following property is the most important one.

\begin{prop}\label{2.Prop.Green}
For any $x\in \Z^2\setminus\0,$ the function $H_x(y)$ belongs to $\CH_0$.
\end{prop}
\begin{proof} We denote $\1_2=(0,1)$,  fix  $x\in \Z^2\setminus\0$ and consider the expression  $d_2H_x(y)=H_x(y+\1_2)-H_x(y).$ By \eqref{2.FundamentalCompl},
\begin{gather}\label{2.calculation.2}\nonumber
  d_2H_x(y)= (2\pi)^{-2}\int\limits_{(-\pi,\pi)^2}\frac{1}{Z(\theta)} \left(e^{i(y+\1_2)\theta}-e^{iy\theta}-e^{i(y-x+\1_2)\theta}+e^{i(y-x)\theta}\right) d\theta \\
=  -(2\pi)^{-2}\int\limits_{(-\pi,\pi)^2}\frac{(e^{i\1_2\theta}-1)(e^{-ix\theta}-1)}{Z(\theta)} \exp (iy\theta) d\theta.
\end{gather}
Considered as a function on the torus $(-\pi,\pi)^2$, the denominator in \eqref{2.calculation.2} has the only zero at the point $\theta=0,\ $ and $Z(\theta)\asymp|\theta|^2$ near $\theta=0$. Both functions $e^{i\1_2\theta}-1$, $e^{-ix\theta}-1$ vanish at $\theta=0$, the main parts being degree one homogeneous in $\theta$. Therefore the function
\begin{equation*}
    K(\theta)=\frac{(e^{i\1_2\theta}-1)(e^{-ix\theta}-1)}{Z(\theta)}
\end{equation*}
is bounded on $(-\pi,\pi)^2$, in particular, it belongs to $L^2((-\pi,\pi)^2)$.  So, the expression \eqref{2.calculation.2} represents the Fourier coefficient of the function $K(\theta)$, and by the Plancherel equality, we have
\begin{equation*}
    d_2H_x(\cdot)\in \ell^2(\Z^2).
\end{equation*}
In a similar way, the function $d_1H_x(y)=H_x(y+\1_1)-H_x(y)$ (where $\1_1=(1,0)$) belongs to
$\ell^2(\Z^2)$. These two facts mean that $\ba_{\Z^2}[H_x]<\infty$ for any $x$.

It remains to show that the function $H_x$ can be approximated in the metric $\ba_G$ by finitely supported functions.
 This is done in the standard way, by setting $H_x^{(m)}(y)=H_x(y)F_m(|y|)$, where $F_m(s)=1, s\le1,$ $F_m(s)=(\log m)^{-1}(\log m -\log s),  s\in(1,m),$ and $F_m(s)=0, s\ge m$. Using that $H_x$ is bounded and the above properties of $d_jH_x$, it is easy to show that $\ba_{\Z^2}[H_x-H_x^{(m)}]\to 0$ as $m\to  \infty$.
\end{proof}
So, the functions $H_x$ and $h_x$, both in $\CH_0$, are the solutions of the same equation. Due to the uniqueness in the construction of $h_x$, we have
\begin{equation*}
    h_x=H_x,
\end{equation*}
and all the properties established in this section for the functions $H_x$ hold now for the functions $h_x$.

\subsection{Sparse sets in $\Z^2$}\label{spZ2}
The definition of a sparse set given in Subsection \ref{sparse potentials}, requires a minor modification in the case of $\Z^2$: the space $\CH$ should be replaced by $\CH_0$. The next statement is an analogue of Proposition \ref{exist}. However, the proof is a little bit different, it uses the asymptotic formula \eqref{2.asympt} and its consequences.
\begin{prop}\label{spsZ2} The graph $\Z^2$ contains a sparse set.
 \end{prop}
 \begin{proof}
For the functions $\wt{h}_x$ introduced in Subsection \ref{Z2Green}, we have
\begin{equation*}
    (\wt{h}_x, \wt{h}_y)=h_x(x)^{-\frac12}h_y(y)^{-\frac12}h_x(y).
\end{equation*}
For a fixed $x\in\Z^2, $ we use the asymptotics \eqref{2.Green.asymp} and \eqref{2.Green.asymp.diag} as $y\to \infty$.
This gives
\begin{equation}\label{2.orthog}
      (\wt{h}_x, \wt{h}_y)\sim\m_x^{-1} (\log|y|)^{-\frac12}(2\log|x|+ \log 8 + 2\g).
\end{equation}
It is clear from \eqref{2.orthog} that $ (\wt{h}_x, \wt{h}_y)\to 0$ as $x$ is fixed and $y\to\infty$.

The rest of the proof uses the same inductive procedure as in Subsection \ref{sps}.
Fix a double sequence $\{\vare_{mn}\},\ m,n\in\N;\ m\neq n,$ of positive numbers, such that
$\vare_{mn}=\vare_{nm}$ and $\sum\vare_{mn}^2<1$.
Suppose that the points $x_1,\ldots,x_{n-1}$ are already chosen ($x_1$ is arbitrary). In view of \eqref{2.orthog},
we can choose a point $x_n\in\Z^2\setminus\0$ such that
\begin{equation}\label{3.ineq}
    |(\wt{h}_{x_m}, \wt{h}_{x_n})|<\vare_{mn},\qquad \forall m< n.
\end{equation}
As a result of this inductive procedure, we obtain an infinite sequence $x_n\in\Z^2\setminus\0$ such that \eqref{3.ineq} is valid for all $m,n\in\N;\ m< n.$
By symmetry, this inequality is satisfied also for $m>n$.

It follows (exactly as it was in Subsection \ref{sps}) that the Gram matrix $\BM=\{(\wt{h}_{x_m}, \wt{h}_{x_n})\}$ is such that $\|\BM-\BI\|_{\GS_2}<1$, and hence,
by Theorem VI.3.3 in \cite{GOKR}, the set $Y=\{x_n\}$ is sparse.
\end{proof}

Now we define a sparse potential on $\Z^2$ as such whose support is a sparse subset.
The spectral properties of the corresponding operator $\BB^{(0)}_{V,\Z^2}$, and hence, those of the
operator $-\D_0-\a V$, are studied in the same way as this was done in Section \ref{comb}.
This leads to the analogues of Theorems \ref{bound-sparse} and \ref{sppot}. The next statement
is a combination of these both results. The proof remains the same and we skip it. We only remind that the operator $\BT$ in the formulation is the one, appearing in the definition of a sparse subset, cf. Subsection \ref{sparse potentials}.

\begin{thm}\label{Z2-sparse}
Let $V\ge 0$ be a sparse potential on $\Z^2$. Then

1. The corresponding Birman-Schwinger operator $\BB^{(0)}_{V,\Z^2}$ is bounded if and only if the function $V$ is bounded, and is compact if and only if $V\to0$.

Moreover, \[ C\|V\|_{\ell^\infty}\le\|\BB^0_{V,\Z^2}\|\le C'\|V\|_{\ell^\infty},\]
where $C=\|(\BI-\BT)^{-1}\|^{-2}$ and $C'=\|\BI-\BT\|^2$. In the case of compactness, the following two-sided estimate, with the same constants $C,C'$, is valid for the eigenvalues $\l_n(\BB^{(0)}_{V,\Z^2})$:
\[ CV_n\le \l_n(\BB^{(0)}_{V,\Z^2})\le C'V_n.\]

\vs
2. If, in addition, the numbers $V_n$
form a moderately varying sequence, then
\begin{equation*}
    \l_n(\BB^0_{V,\Z^2})\sim V_n.
\end{equation*}
\end{thm}

\section{operators on metric graphs}\label{metric}
Let $\G$ be a metric graph, with edge lengths $l_e$. As in \cite{RS10}, we associate with $\G$ the combinatorial graph $G=G(\G)$, with the same sets $\CV,\CE$ and the same connection relations, and we suppose the general conditions on $G$ formulated in the beginning of Section \ref{comb} to be fulfilled. To any edge $e$ of $G(\G)$ we assign the weight $g_e=l_e^{-1}$. The basic Hilbert space is now $L^2(\G)$, with respect to the measure induced by the Lebesgue measure on the edges.

The Sobolev space $\CH^1(\G)$ consists of all continuous functions $\varf$ on $\G$, such that $\varf\in H^1(e)$ on each edge and $\int_\G(|\varf'|^2+|\varf|^2)dx<\infty$.
The operator $-\D$ on $\G$ is defined via its quadratic form $\ba_\G[\varf]=\int_\G(|\varf'|^2dx$, with the form-domain $\CH^1(\G)$. The Laplacian $\D$ acts as $\varf''$ on each edge, and the functions from $\dom \D$ satisfy the natural (Kirchhoff) conditions at each vertex.

 Here we list our assumptions about the graph $\G$.
\begin{assump} \label{ass3.1}One of the following holds:\\
(i) The edge lengths $l_e$ of the graph $\G$ are bounded from above, the corresponding combinatorial graph $G(\G)$ has global dimension $D>2$ and contains a sparse subset, say $Y$. \\
(ii) The graph $G(\G)$ is $\Z^2$.
\end{assump}

To define the Birman-Schwinger operator, we need to consider two cases of Assumption \ref{ass3.1} separately.

In the case (i),  by theorem 4.1 in \cite{RS10}, the graph $\G$ has the same global dimension $D$. Therefore the inequality of the form \eqref{varop}, with      $\ba_\G$ replacing $\ba_G$,  holds for all functions in $\CH^1(\G)$ with compact support, and consequently the space $\CH(\G)$, the closure of the set of compactly supported functions in the metric $\ba_\G$, is a space of functions (see \cite{RS10}, where the details are presented.)

In the case (ii), i.e., $G(\G)=\Z^2$, one should consider the subspace $\CH^1_0(\G)$ in $\CH^1(\G)$, consisting of functions, vanishing at the vertex $\0$. A Hardy type inequality, derived easily from \eqref{2.Hardy}, implies that the space $\CH_0(\G)$, the closure of  the set of compactly supported functions in $\CH^1_0(\G)$ in the metric $\ba_{\G}$, is a space of functions as well.

Further on, to unify notations, we suppress the subscript $0$ when dealing with $\CH_0(\G)$, so, for $G=\Z^2$, $\CH(\G)$ means $\CH_0(\G)$.

 An analogue of the quadratic form \eqref{bbV} is given by
\begin{equation}\label{bbVmetr}
    \bb_{V,\G}[\varf]=\int_\G V|\varf|^2dx,
\end{equation}
where $V\ge 0$ is a function in $L^1(\G)$.
We denote by $\BB_{V,\G}$  the operator in $\CH(\G)$, generated by the
quadratic form \eqref{bbVmetr}.

   We will construct  potentials $V$ on $\G$, whose support is a vicinity (in $\G$) of the set $Y$.
By using such potentials, we will be able to prove the following statements that can be considered as analogues of theorems \ref{bound-sparse} and \ref{sppot}.

\begin{thm}\label{cont2.1}
Let for a graph $\G$ Assumption \ref{ass3.1} be satisfied.
Let a sequence $p_n\searrow0$ , such that $n^{1/2}p_n\to0$, be given. Then there exists a function $V\in L^1(\G),\ V\ge 0,$ such that
\begin{equation*}
    C p_n\le \l_n(\BB_{V,\G})\le C'p_n,\qquad \forall n\in\N,
\end{equation*}
where the constants $C_0, C\in(0,\infty)$ do not depend on the sequence $\{p_n\}$.
\end{thm}
\vs

\begin{thm}\label{cont2.3}
Suppose in addition that the sequence $p_n$ is moderately varying. Then the function $V\in L^1(\G),\ V\ge 0,$ can be chosen in such a way that
\begin{equation*}
    \l_n(\BB_{V,\G})\sim p_n.
\end{equation*}
\end{thm}

Below we restrict ourselves to the proof of \thmref{cont2.3}. The proof of \thmref{cont2.1} is much simpler and can be easily reconstructed by the analogy.
As in \cite{RS10}, section 4.3, for the analysis of the operators $\BB_{V,\G}$ we use of the orthogonal (in $\CH(\G))$ decomposition
\begin{equation}\label{1a}
    \CH(\G)=\CH_{\pl}\oplus\CH_\CD.
\end{equation}
Here the subspace $\CH_{\pl}$ is formed by  functions $\varf\in\CH(\G)$, which are linear on each edge $e$, and $\CH_{\CD}$ is formed by  functions vanishing at each vertex $v$. Below we denote by $\BB_{V,\G;\pl},\  \BB_{V,\G;\CD}$ the operators in these subspaces, generated by the quadratic forms $\bb_{V,\G}\res\CH_{\pl}$ and
$\bb_{V,\G}\res\CH_{\CD}$ respectively. Given a function $\varf\in\CH(\G)$, we denote by $\varf_{\pl},\ \varf_\CD$ its components in the decomposition \eqref{1a}.

\noindent{\it Proof of \thmref{cont2.3}.} Let $\wh{V}$ be a sparse potential on the combinatorial graph $G(\G)$, supported on a sparse set $Y\subset \CV$, such that
$\wh{V}(v_n)=p_n\searrow 0$, $v_n\in Y$. Along with the operator $\BB_{\wh{V},G(\G)}$, let us consider the operator
$\wh\BB_{\wh{V},\G}$ in the space $\CH(\G)$, associated with the quadratic form
\begin{equation*}
    \wh\bb_{\wh{V},\G}[\varf]=\sum_np_n|\varf(v_n)|^2.
\end{equation*}
This quadratic form vanishes on $\CH_{\CD}$, and
\begin{equation}\label{2a}
    \wh\bb_{\wh{V},\G}[\varf]=\bb_{\wh{V},G(\G)}[\varf\res\CV].
\end{equation}
It follows that the non-zero spectra of the operators $\wh\BB_{\wh{V},\G}$ and $\BB_{\wh{V},G(\G)}$
 coincide. So, if we allow such point-supported potentials, then all the results  of Section
\ref{sparse potentials} carry over to the metric graphs.

Now we show how to construct a "genuine" potential $V\in L^1(\G)$ that possesses the desired
  properties. Suppose for simplicity that $Y=Y_{\wh{V}}$ contains no neighboring vertices. This
   can be always achieved by a thinning down the original sparse set.
  Choose a sequence $\vare_n>0$, such that $\vare_n<\min\limits_{e\ni v_n} l_e$. Further assumptions about the behavior of $\vare_n$ will be imposed later. For the calculations
  below, it is convenient to denote by $\wh\CE$ the set of all edges, such that one of its ends lies in $Y$.

Define now the potential $V$ on $\G$, as follows:
\begin{equation}\label{pot}
    V(y)=\begin{cases}\frac{p_n}{\vare_n \deg v_n} & {\rm{if}}\ \dist(y,v_n)<\vare_n\  {\rm{for\ some}} \ n,\\
    0 & \ {\rm{otherwise.}}
    \end{cases}
\end{equation}
We have
\begin{equation}\label{qform}
    \bb_{V,\G}[\varf]=\bb_{V,\G}[\varf_{\pl}]+\bb_{V,\G}[\varf_\CD]+2\re \bb_{V,\G}[\varf_{\pl},\varf_\CD].
\end{equation}
We shall inspect each term separately.

Any edge $e\in\wh{\CE}$ can be written as $e=(v_n,v'_n)$ where $v_n\in Y$ and $v'_n\notin Y$. The vertex $v'_n$
that corresponds to the vertex $v_n$ is not unique, so that here we have some slight abuse of notation, however what is important is that $v'_n$ is determined uniquely by $e$.
 We identify any such edge with the interval $(0,l_e)$, so for $y\in e$ we have
\[ \varf_{\pl}(y)=l_e^{-1}\left(\varf(v_n)(l_e-y)+\varf(v'_n)y\right).\]
From here we derive
\begin{equation*}
    \bb_{V,\G}[\varf_{\pl}]=\sum_{e\in\wh\CE}\frac{p_n}{l_e^2\vare_n\deg\,v_n}
\int_0^{\vare_n}|\varf(v_n)(l_e-y)+\varf(v'_n)y|^2dy.
\end{equation*}
In each summand the term that will be shown to be dominant is
\[I_e[\varf]= \frac{p_n}{\vare_n\deg\,v_n}
\int_0^{\vare_n}|\varf(v_n)|^2dy=\frac{p_n}{\deg\,v_n}|\varf(v_n)|^2.\]
So,  by \eqref{2a}
\[ \sum_{e\in\wh\CE}I_e[\varf]=\bb_{\wh{V},G(\G)}[\varf_{\pl}].\]

To estimate the difference $\bb_{V,\G}[\varf_{\pl}]-\bb_{\wh{V},G(\G)}[\varf_{\pl}]$
we use the elementary inequality $\bigl| |A|^2-|B|^2\bigr|\le|A-B|(|A|+|B|)$, to obtain
\begin{gather*}
\bigl||\varf(v_n)(l_e-y)+\varf(v'_n)y|^2-|\varf(v_n)|^2l_e^2\bigr| \\ \le y|\varf(v_n)-\varf(v'_n)|
\left((2l_e-y)|\varf(v_n)|+y|\varf(v'_n)|\right)\le 2l_e y(|\varf(v_n)|^2+|\varf(v'_n)|^2).
\end{gather*}
Now the summation over all edges $e\in\wh{\CE}$ gives
\begin{equation}\label{3a}
    \bigl| \bb_{V,\G}[\varf_{\pl}]-\bb_{\wh{V},G(\G)}[\varf_{\pl}]|
    \le\sum_{e=(v_n,v'_n)\in\wh{\CE}}R(e,\vare_n;\varf\res\CV)
\end{equation}
where
\[ R(e,\vare_n)=\frac{p_n\vare_n}{l_e\deg\,v_n}(|\varf(v_n)|^2+|\varf(v'_n)|^2).\]
Choosing $\vare_n$ decaying fast enough, we can grant an arbitrarily fast decay of the eigenvalues of the operator corresponding to  the quadratic form in the right-hand side
of \eqref{3a}.  This yields that
\begin{equation*}
    \l_n(\BB_{V,\G;\pl})\sim p_n.
\end{equation*}

Now we go over to the quadratic form  $\bb_V[\varf_\CD]$. The associated operator $\BB_{V,\G;\CD}$ can be identified with the direct orthogonal sum of operators $\BB_{V,e},\ e\in\wh{\CE}$ acting in the spaces $H^{1,0}(e)$, and the Rayleigh quotient for each operator $\BB_{V,e}$ is
\[ \frac {\int_e V(y)|u(y)|^2dy}{\int_e|u'(y)|^2dy},\qquad u\in H^{1,0}(e),\]
where the weight function $V(y)$ is defined by (4.4).
For the eigenvalue estimates of such operators we use the following
\begin{prop}\label{cal}
Let $V(y)\ge 0$ be a monotone function on an interval $e=(0,a),\ a\le\infty$, such that $\int_e\sqrt{V(y)}dy<\infty$. Then for any $\l>0$
\[ n(\l,\BB_{V,e})\le \frac2{\pi}\l^{-1/2}\int_e\sqrt{V(y)}dy.\]
\end{prop}
\begin{proof} For $a=\infty$, due to the Birman -- Schwinger principle, this is an equivalent reformulation of the well known Calogero estimate (\cite{Calogero}; see also \cite{Reed-Simon 4}, Theorem XIII.9(b)). The case $a<\infty$ follows from here by the standard variational argument.
\end{proof}

Applying Proposition \ref{cal} to each edge $e=(v_n,v'_n)\in\wt{\CE}$, we obtain
\begin{equation}\label{estim1}
    n(\l,\BB_{V,\G;\CD})=\sum_{e\in\wh{\CE}}n(\l,\BB_{V,e;\CD})\le \frac2{\pi}\l^{-1/2}\sum_{e\in\wh{\CE}}\left(\frac{p_n\vare_n}{deg\, v_n}\right)^{1/2}.
\end{equation}
 If the series on the right converges, then $n(\l,\BB_{V,\G;\CD})=O(\l^{-1/2})$ (or, equivalently, $\l_j( \BB_{V,\G;\CD})\le Cj^{-2}$), and hence, the operator $\BB_{V,\G;\CD}$ does not contribute to the asymptotics of any greater order.

It remains for us to inspect the last term in \eqref{qform}. It is equal to
\begin{equation*}
    \Om[\varf]=2\re\sum_{e=(v_n,v'_n)\in\wh{\CE}}\frac{p_n}{l_e^2\vare_n\deg v_n}
    \int_0^{\vare_n}(\varf(v_n)(l_e-y)+\varf(v_n')y)\overline{\varf_\CD(y)}dy.
\end{equation*}
The main contribution to this sum is given by the similar expression, say $\Om_0[\varf]$, with the first factor in the integrand replaced by $l_e\varf(v_n)$. For the spectral estimates we choose one more sequence $\g_n$, vanishing as $n\to\infty$.

Each term in $\Om_0[\varf]$
does  not exceed
\begin{gather*}
    \frac{p_n}{l_e\vare_n\deg\,v_n}|\varf(v_n)|\int_0^{\vare_n}|\varf_\CD(y)|dy\\
    \le \g_n p_n|\varf(v_n)|^2+\frac{p_n}{4\g_n(\deg v_n)^2}\vare_n
    \int_0^{\vare_n}|\varf_\CD(y)|^2dy.
\end{gather*}
Here the first terms on the right correspond to the quadratic form that by \thmref{bound-sparse} generates an operator whose eigenvalues, due to the factor $\g_n$, behave as $o(p_n)$. The second terms are quadratic forms of the same type as the ones considered above, for the operator $\BB_{V,\G;\CD}$, with the same weight function. Applying the estimate \eqref{estim1} (and again, choosing $\vare_n$ decaying fast enough), we see
that the corresponding estimate is of the Weyl type, and so, this terms do not affect the spectral asymptotics of any order greater than $O(\l^{-1/2})$.

The estimates for the quadratic form $\Om[\varf]-\Om_0[\varf]$ are much easier. We have
\begin{gather*}
    |\Om[\varf]-\Om_0[\varf]|\le2\sum_{e=(v_n,v'_n)\in\wh{\CE}}\frac{p_n|\varf(v'_n)-\varf(v_n)|}
    {l_e^2\deg\,v_n}\int_0^{\vare_n}|\varf_\CD(y)|dy
\end{gather*}
Each term here can be estimated through
\begin{equation*}
    \frac{p_n^2(|\varf(v_n')|^2+|\varf(v_n)|^2)\vare_n}{l_e^4(\deg v_n)^2}+ \int_0^{\vare_n}|\varf_\CD(y)|^2dy.
\end{equation*}
The first terms here give a quadratic form whose eigenvalues decay as fast as we wish, provided the sequence $\vare_n$
is chosen in an appropriate way.  The same is true for the second terms, after we apply the estimate \eqref{estim1}.$\Box$

\end{document}